\documentclass[a4paper,reqno,12pt]{amsart}
\usepackage{amsfonts,amssymb,amsthm,bbm,amscd}
\usepackage{graphicx,epsfig}
\usepackage{perpage}
\usepackage{url}
\usepackage{color}
\usepackage{epstopdf}
\usepackage{mathabx}
\usepackage{amsmath}
\usepackage{breqn}
\usepackage{amssymb}
\usepackage{amsfonts}
\usepackage{float}
\usepackage{framed}
\usepackage{subfigure}
\usepackage{mathtools}
\usepackage{enumerate}
\usepackage{hyperref}
\usepackage{microtype}
\usepackage{comment}
\usepackage{mathabx}

\usepackage{palatino}
\usepackage[T1]{fontenc}
\usepackage[mathscr]{eucal}
\usepackage{dsfont}
\usepackage{fullpage}

\usepackage[font=small,labelfont=bf]{caption}
\usepackage{subfigure}
\usepackage{tikz}
\usetikzlibrary{calc,patterns}

\usepackage{hyperref}
\hypersetup{
colorlinks=true,
linktocpage=true,
pdfstartview=FitH,
breaklinks=true,
pdfpagemode=UseNone,
pageanchor=true,
pdfpagemode=UseOutlines,
plainpages=false,
bookmarksnumbered,
bookmarksopen=false,
bookmarksopenlevel=1,
hypertexnames=true,
pdfhighlight=/O,
pdftitle={},
pdfauthor={},
pdfsubject={},
pdfkeywords={},
pdfcreator={pdfLaTeX},
pdfproducer={LaTeX with hyperref}
}

\numberwithin{equation}{section} 
\usepackage[sort&compress,capitalize,nameinlink]{cleveref}
\crefname{app}{Appendix}{Appendices}

\crefrangeformat{equation}{\upshape(#3#1#4)\textendash(#5#2#6)}

\theoremstyle{plain}
\newtheorem{theorem}{Theorem}

\newtheorem{lemma}[theorem]{Lemma}
\newtheorem{proposition}[theorem]{Proposition}

\newtheorem{remark}[theorem]{Remark}

\newtheorem{cond}[theorem]{Condition}

\numberwithin{theorem}{section}

\allowdisplaybreaks

\begin{document}

\title[Estimate of the exit time for the Long Range Ising model on random regular graphs]{Estimate of the exit time for the Long Range Ising model \\on random regular graphs}

\date{}

\author{Vanessa Jacquier$^\star$}

 \thanks{$^\star$vanessa.jacquier@unipd.it. University of Padova, via Luigi Luzzatti 4, 35121, Padova, Italy}

\begin{abstract}
We investigate the metastable behavior of the long-range Ising model on random regular graphs under Glauber dynamics at low-temperature. We estimate the energy barrier and exit time from the metastable state using a nontrivial path-wise approach that explicitly accounts for the spatial decay of the interactions and the structural properties of the graph, such as the Cheeger constant and known estimates of the diameter. Our results generalize those of Dommers \cite{dommers2017metastability} for the short-range case, providing a unified framework for understanding metastability in systems with long-range interactions. \\

\noindent   {\it MSC2020}  {\it Subject classifications:} 82B20, 82C20, 82C26, 60K35, 60J10, 05C80 \\
    
\end{abstract}

\maketitle


\section{Introduction}
The Ising model is one of the most interesting models to study not only the phase transitions, but also the magnetism and the cooperative phenomena in various physical systems.
The standard Ising model describes a system of spins arranged on a lattice, where each spin can take value $+1$ or $-1$. The interactions among these spins determine the behavior of the system, leading to phenomena such as ferromagnetism when the spins align. These interactions are typically short-range, meaning that a spin interacts only with its nearest neighbors on the lattice. 

The long-range Ising model generalizes the traditional Ising model by allowing spins to interact with each other over longer distances, not just with their nearest neighbors. The model is generally described by the following Hamiltonian function:
\begin{equation}
H(\sigma)=-\sum_{x \neq y} \mathcal{J}_{x,y}  \sigma_x \sigma_y -h \sum_{x} \sigma_x, \notag
\end{equation}
 where $\mathcal{J}_{x,y}$ represents the interaction strength between the couple of spins at sites $x$ and $y$. This interaction term usually takes two forms 
\begin{align}
    \mathcal{J}_{x,y}=
    \begin{cases}
        J/||x-y||^\lambda & \text{ polynomial decay,} \notag \\
        J/\lambda^{ ||x-y||} &\text{ exponential decay,}
    \end{cases}
\end{align}
where $J$ is a positive constant and $\lambda>1$ is a parameter that controls the range of the interaction. We note that for $\lambda$ large, the model reduces to the short-range Ising model, while for small $\lambda$ the interactions become long-range.

The introduction of long-range interactions leads to richer and more complex behavior, which is more challenging to analyze than the short-range version. Indeed the critical properties of the phase transitions change according to the value of $\lambda$ and in some cases the nature of these interactions can determine the emergence of new phases of matter or critical points that are not present in the short-range model.
Moreover, this kind of interaction makes the study non-local, meaning that a spin can influence another spin far away in the lattice, and this non-locality complicates both analytical computations and numerical simulations.
The reason why the study of this model is interesting lies in the fact that it has several real applications. Indeed dipolar magnets, gravitational systems, and certain types of neural networks exhibit interactions that decay with distance but not as rapidly as in the short-range case. For this,
the long-range Ising model provides a framework for studying the systems where short-range interactions are an oversimplification. So, it can be defined as a powerful extension of the classic Ising model that captures the effects of interactions extending over longer distances, leading to novel and complex physical behaviors. 

In this paper, we investigate the metastable behavior of the long-range Ising model on random regular graph, extending the results presented in \cite{dommers2017metastability}.
Metastability is widespread phenomenon observed in various fields as physics, chemistry, biology and economics. Formally, it refers to the scenario in which a system remains in a non-equilibrium state, the so-called \emph{metastable state}, for an exponential long time before eventually transitioning to a more stable equilibrium state. 
Our analysis focuses on the low-temperature regime, where this transition is typically preceded by the formation of a critical mesoscopic configuration, triggered either by a spontaneous fluctuation or an external perturbation.
In particular, when the system is initiated in the metastable phase, it undergoes a slow transition toward the stable phase, requiring it to overcome an energy barrier to reach the equilibrium. 

Despite significant interest, rigorous results on the metastable behavior of long-range Ising models remain limited, and only a few studies provide quantitative insights into nucleation and transition processes.

In \cite{van2019nucleation}, the authors study the one-dimensional long-range Ising model, which, unlike its short-range counterpart, undergoes a phase transition already in one dimension, a phenomenon that persists even in rather rapidly decaying fields (see \cite{cassandro2005geometry, dyson1972existence,bissacot2018contour, littin2017quasi}). They analyze the evolution of the model under Glauber dynamics and prove that, for sufficiently small external magnetic field, the homogeneous state $\textbf{-1}$ is the unique metastable state, also providing an estimate of the transition time from this state to the stable state $\textbf{+1}$.

In higher dimensions, rigorous results on the metastable behavior of the long-range Ising model are still rare. Nevertheless, \cite{jacquier2024isoperimetric} makes an important contribution by identifying the typical configurations that are likely to trigger nucleation in the biaxial long-range Ising model, providing a first step toward a better understanding of the metastable behavior. 

Building on this perspective, the recent work \cite{lazarides2025apparent} shows that systems with weakly long-range interactions can exhibit states that appear bistable for extremely long times. This occurs because the nucleation of destabilizing droplets is highly suppressed, so metastable states can persist effectively indefinitely. Consequently, the observed apparent bistability can be understood as a manifestation of prolonged metastability induced by long-range interactions.

While much work has focused on lattices, understanding metastability on complex graph structures remains largely unexplored. In this work, we examine the metastable behavior of the long-range Ising model evolving under Glauber dynamics on random $r$-regular graphs, the graphs in which each vertex has the same number $r$ of edges (degree) and these edges are assigned randomly. Our results provides an estimate for the energy barrier and the exit time of the system, emphasizing the pivotal role of the graph structure in determining these quantities.
Indeed, the regularity of these graphs can affect metastability. For instance, the uniform degree distribution can make the system more homogeneous, potentially leading to longer metastable states, as every vertex has the same local environment. 
However, the combination of randomness and regularity in the graph structure makes it difficult to derive exact analytical results. For this reason we can only determine an interval for the value of the energy barrier rather than a precise estimate.

Specifically, our results extend those presented in \cite{dommers2017metastability}, where the energy barrier was estimated for the short-range Ising model on random regular graphs. In particular, we show that when the long-range interaction is defined as an exponential function of the graph distance, our findings are consistent with those in \cite{dommers2017metastability}, up to a constant factor depending on the vertex degree $r$. However, when the interaction decays as a fractional power of the distance (with exponent $\lambda>1$), the approximation becomes less accurate. To establish these results, we not only employed standard techniques from metastability theory, such as the path-wise approach, but also leveraged key structural properties of the graphs, including the Cheeger constant and the graph diameter, as well as known estimates of these quantities from the literature.

Other studies regarding metastability of Ising models on graphs highlight the challenges posed by random or non-uniform connections, where both thermal fluctuations and the graph structure can delay convergence to the globally stable state.

In \cite{dommersnardi2017metastability}, the authors obtain a similar estimate of the energy barrier when they consider the behavior of the Ising model on a random multi-graph known as the \emph{configuration model}, evolving under Glauber dynamics at low temperature.

The metastable behavior of a variant of the Ising model on random graphs has also been analyzed in \cite{hollander2021glauber} and \cite{bovier2021metastability}, focusing on the \emph{Erdős–Rényi random graph} with a fixed edge retention probability for the Curie–Weiss model.

\subsubsection*{Organization of the paper}
The paper is organized as follows. 
In Section \ref{sec:model} we introduce the model and state the main results. Sections \ref{sec:proof1} and \ref{sec:proof1} are devoted to their proofs: we first estimate the stability level of the homogeneous state $\mathbf{-1}$, and then that of the configurations different from $\textbf{-1}$ and $\textbf{+1}$. Finally, in Section \ref{sec:auxiliary_lemma}, we establish the auxiliary lemmas employed in the proofs of the main results. 

\section{Model and results}
\label{sec:model}
In this Section, we introduce the model and the main results. In particular, in Section \ref{subsec:graph} we provide the definitions and the relevant known results on random regular graphs. In Section \ref{subsec:longrange_Ising_Glauber}, we define the long-range Ising model and Glauber dynamics, while in Section \ref{subsec:metastability} we describe the metastability phenomenon. In Section \ref{subsec:main_results}, we present the main results of this work.

\subsection{Random regular graphs}
\label{subsec:graph}
For any $r\in\mathbb{N}$ and $n>r$ with $nr$ even, let $G_n=(V_n,E_n)$ be a random $r$-regular graph with $n$ vertices, i.e., $G_n$ is selected uniformly at random from the set of all simple $r$-regular graphs with $n$ vertices. 

We consider $v,u$ in the set of vertices $V_n=\{1,...,n\}$ and we define the \emph{distance} $d(u,v)$ between $v,u$  as the number of edges in a shortest path connecting them. 

The \emph{diameter} of $G_n$ is the length of the shortest path between the most distanced vertices, i.e. 
\begin{align}
    \text{diam}(G_n):=\max_{u,v \in V_n} d(u,v).
\end{align}
For simplicity, in the rest of the paper, we will denote it by $d$. 
In \cite[Theorem 1]{bollobas1982diameter}, the authors find the following estimate for the diameter of a $r$-regular random graph $G_n$.
\begin{theorem}[Upper bound for the diameter]
Let $r \geq 3$ be an integer number, $\varepsilon \in (0,1)$ and $d=d(n)$ be the least integer such that
 \begin{align}
     (r-1)^{d-1} \geq (2+\varepsilon) rn \log{n}. \notag
 \end{align}
Then with high probability (w.h.p.)
\footnote{We say that an event $\mathcal{E}_n$ holds \emph{with high probability (w.h.p)} if $\lim_{n \to \infty} \mathbb{P}(\mathcal{E}_n)=1$.
}
the r-regular graph $G_n$ has diameter at most $d$, that is $\text{diam}(G_n) \leq d$ with
\begin{align}\label{upperbound_diameter}
    d:= \Big \lceil \frac{1}{\log{(r-1)}} \Big ( \log{((2+\varepsilon)r)}+\log{(n \log{n})} \Big ) +1 \Big \rceil.
\end{align}
\end{theorem}
\noindent

Moreover, for each subset $A \subset V_n$, we denote by $|A|$ the cardinality of $A$ and by $A^c$ its complement.
Furthermore, we say that $A$ is \emph{connected} if and only if for any $x\neq x'\in A$ there exists a sequence $x_1,x_2,\dots,x_m$ of vertices in $A$ such that $x_1=x$, $x_m=x'$, and $(x_k,x_{k+1}) \in E_n$ for any $k=1,\dots,m-1$.

We define the set of edges joining $A \neq \emptyset$ with $A^c$, i.e.
\begin{align}
\partial_e A= \{ (x,y) \in E_n \, | \, x\in A, \, y\not \in A\}.
\end{align}
We extend the previous definition to all vertices of $A$ in the following way. We consider the set of couples of vertices $(x,y)$ at distance $i$ such that one belongs to $A \neq \emptyset$ and the other to $A^c$, i.e.
\begin{align}\label{def:external_boundary}
    \partial_e^{(i)} A=\{ (x,y) \in V_n \times V_n \, | \, x\in A, \, y\not \in A, \, \text{ and } d(x,y)=i \},
\end{align}
for $i=1,...,d$.
We note that $\partial_e^{(1)} A \equiv \partial_e A$ and we note that we may have $\partial_e^{(i)} A= \emptyset$ if there are not vertices of $A^c$ at distance $i$ from some vertices in $A$. Moreover, we observe that $|\partial_e^{(i)}A^c|=|\partial_e^{(i)}A|$, since the two sets contain the same pairs with a different order.
A fundamental property of the $r$-regular graphs that we will often use is the following inequality. For $i=1,...,d$,
\begin{align}\label{inequality_lower_graph}
    |\partial_e^{(i)} A| \leq (r-1)|\partial_e^{(i-1)} A| \leq \cdots \leq (r-1)^{i-1}|\partial_e A|.
\end{align}

Finally, we define the \emph{isoperimetric number} or \emph{Cheeger constant} of a graph $G_n$ as follows
\begin{align}\label{def:isoperimetric_number}
    i_e(G_n)= \min_{A \subset V_n, \,  |A| \leq n/2} \frac{|\partial_e A|}{|A|}.
\end{align}
The Cheeger constant measures how well-connected a graph is, quantifying the minimal \emph{edge boundary} relative to the size of a vertex subset. In this context, it plays a crucial role in determining energy barriers and transition times, as graphs with higher Cheeger constants tend to have more robust connectivity that can influence the stability of the system.
In \cite[Theorem 1 and Corollary 1]{bollobas1988isoperimetric}, the authors find the following lower bound for the Cheeger constant.
 \begin{theorem}[Lower bound for the Cheeger constant]\label{thm:lowerbound_isop}
 Let $r \geq 3$ be an integer number and $\varepsilon \in (0,1)$ be such that
 \begin{align}
     2^{\frac{4}{r}} <(1-\varepsilon)^{1-\varepsilon}(1+\varepsilon)^{1+\varepsilon}. \notag
 \end{align}
Then w.h.p. 
the r-regular graph $G_n$ has Cheeger constant at least $(1-\varepsilon) \frac{r}{2}$ and, in particular,
\begin{align}\label{upper_bound_CHeeger}
    i_e(G_n) \geq \frac{r}{2}- \sqrt{\log{2}} \sqrt{r}.
\end{align}
 \end{theorem}
Moreover, the authors of \cite{alon1997edge} provide an upper bound for the Cheeger constant of $G_n$. In particular, they obtain that for some constant $C>0$
\begin{align}\label{lower_bound_CHeeger}
   i_e(G_n) \leq \frac{r}{2}-C \sqrt{r},
\end{align}
by proving \cite[Theorem 1.1]{alon1997edge} that we report below.
  \begin{theorem}[Upper bound for the Cheeger constant]
  There exists a constant $C>0$ such that for any $r$-regular graph $G_n$ with $n \geq 40 r^9$ there exist a set $U \subset V_n$ with $|U|= \lfloor \frac{n}{2} \rfloor$ such that
\begin{align}
    |\partial_e U| \leq \Big ( \frac{r}{2} - C\sqrt{r} \Big ) |U|.
\end{align}
  \end{theorem}
  
\subsection{Long Range Ising model and Glauber dynamics}
\label{subsec:longrange_Ising_Glauber}
 Let $n\in \mathbb{N}$ and $G_n$ be a random r-regular graph. We associate to each vertex $x \in V_n$ a spin variable $\sigma_x \in \{-1, +1\}$ and we denote by $\mathcal{X}:=\{-1, +1\}^{V_n}$ the \emph{configuration space} (or \emph{state space}) of the model. Moreover, given a set $A \subseteq V_n$, we denote by $\sigma_A$ the configuration restricted to the set $A$.

In this paper, we analyze the Long-Range Ising model with positive external magnetic field $h$ and pair interaction $\mathcal{J}: \mathbb{N} \to \mathbb{R}$ positive and decreasing. In particular, we consider two different interactions: 
\begin{align}\label{interactions}
    \mathcal{J}(n)= J r^{1-n} \qquad \text{ and } \qquad \mathcal{J}(n)= J n^{-\lambda},
\end{align} 
where $J>0$ and $\lambda>2$ are two constants and the variable $n$ indicates the distance between the two vertices of the couple. We observe that for $n=1$ they reduce to the special case of the short-range ferromagnetic interaction described by the coupling constant $J$. Thus, our Hamiltonian function is define as follows 

\begin{equation}\label{hamiltonian_function}
H(\sigma)=-\sum_{x,y \in V_n, \,\, x \neq y} \mathcal{J}(d(x,y)) \sigma_x \sigma_y -h \sum_{x \in V_n} \sigma_x.
\end{equation}

We denote by \textbf{+1} and \textbf{-1} the homogeneous states in which all the spins are pluses and minuses respectively, and given $\sigma \in \mathcal{X}$ we denote by $\Delta H(\sigma)$ the following energy difference 
\begin{align}\label{def:delta_H}
  \Delta H(\sigma)=H(\sigma)-H(\textbf{-1}).
\end{align}
 Thus, 
\begin{equation}\label{energy_plus_minus}
   \Delta H(\textbf{+1}) = - 2h n  \qquad \text{ and } \qquad 
   \Delta H(\textbf{-1}) = 0,
\end{equation}
and it is easy to see that $\textbf{+1}$ is the global minumum for the Hamiltonian function.
 We note that, by Definition \eqref{def:delta_H}, we can use the number of pluses to characterize the configurations. 
  For this reason, we introduce the manifold $\nu_m$ that denotes the set of configurations with $m$ pluses, i.e.
  \begin{align}\label{def:manifold}
      \nu_m= \Big \{ \sigma \in \mathcal{X} \, \Big | \, \sum_{x\in V_n} \frac{\sigma(x)+1}{2}=m \Big \}.
 \end{align}

We study the evolution of the system under \emph{Glauber dynamics}. Specifically, we consider a Markov chain  $(\sigma_t)_{t \in \mathbb{N}}$ on $\mathcal{X}$ defined via the so called \emph{Metropolis Algorithm} with
transition probabilities between two configurations $\sigma$ and $\eta$ are given by
\begin{equation}\label{def:glauber}
    p(\sigma, \eta)=\left\{
    \begin{array}{ll}
    q(\sigma,\eta) e^{-\beta[H(\eta) -H(\sigma)]_+} &\qquad \text{if } \sigma \neq \eta, \\
    1- \sum_{\eta \in \mathcal{X}} q(\sigma,\eta) e^{-\beta[H(\eta) -H(\sigma)]_+} &\qquad \text{if } \sigma =\eta,
    \end{array}
    \right.
\end{equation}
where $[\cdot]_+$ denotes the positive part and $q(\sigma,\eta)$ is a connectivity matrix
independent of $\beta$, defined as
\begin{equation}
    q(\sigma,\eta)= \left\{
    \begin{array}{ll}
    \frac{1}{n} & \qquad \textrm{ if } \exists \, x\in V_n: \sigma^{(x)}=\eta, \\
    0 & \qquad \textrm{ otherwise, }
    \end{array}
    \right.
\end{equation}
with
\begin{equation}
    \sigma^{(x)}(z)= \left\{
    \begin{array}{ll}
        \sigma(z) & \;\;\textrm{ if } z \neq x,\\
        -\sigma(x) & \;\;\textrm{ if } z = x.
        \end{array}
    \right.
\end{equation}

It is possible to check that this dynamics is \emph{reversible} with respect to the Gibbs measure
\begin{equation}\label{def:gibbs}
    \mu(\sigma)=\frac{e^{-\beta H(\sigma)}}{\sum_{\eta \in \mathcal{X}} e^{-\beta H(\eta)}},
\end{equation}
where the parameter $\beta:=\frac{1}{T} >0$ is the inverse temperature. Formally, $(\sigma_t)_{t \in \mathbb{N}}$ is an ergodic-aperiodic Markov chain on $\mathcal{X}$  
satisfying the detailed balance condition
\begin{equation}\label{reversibility}
    \mu(\sigma)p(\sigma, \eta)=\mu(\eta)p(\eta, \sigma).
\end{equation}

\subsection{Metastability}
\label{subsec:metastability}

When studying the problem of metastability, we are interesting to find the first arrival of the process $(\sigma_t)_{t \in \mathbb{N}}$ to the set of the \emph{stable states}, corresponding to the set of absolute minima of the Hamiltonian function, starting from an initial local minimum. 
The local minima can be distinguished by their stability level, i.e., the height of the energy barrier separating them from lower energy states. 

More precisely, let $\mathcal{X}^s$ be the set of global minima of the energy, and we refer to it as the set of the stable states (or ground states).
Let $\omega=\{\omega_1,\ldots,\omega_m\}$ be a finite sequence of configurations in $\mathcal{X}$, where $\omega_{k+1}$ is obtained from $\omega_k$ by a single spin flip for each $k=1,...,m-1$. We call $\omega$ a \emph{path} with starting configuration $\omega_1$ and final configuration $\omega_m$ and we denote the set of all these paths as $\Theta(\omega_1,\omega_m)$. 
We call \emph{communication height} between two configurations $\sigma$ and $\eta$ the minimum among the maximal values of the energy along the paths in $\Theta(\sigma,\eta)$, i.e.,
\begin{equation}\label{minmax}
\Phi(\sigma,\eta):=\min_{\omega\in\Theta(\sigma,\eta)}\max_{\zeta \in \omega} H(\zeta).
\end{equation}
See Figure \ref{fig:comm_height}. 
\begin{figure}[!htb]
        \begin{center}
        \includegraphics[scale=1.5]{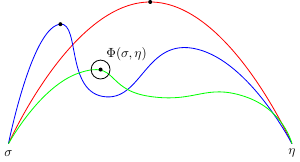}
        \end{center}
        \caption{A schematic representation of the communication height between two configurations $\sigma$ and $\eta$.}
        \label{fig:comm_height}
    \end{figure} 

Given a configuration $\sigma$, we consider the set $\mathcal{I}_\sigma$ of all configurations $\eta$ with $H(\eta)<H(\sigma)$. Note that $\mathcal{I}_\sigma =\emptyset$ if $\sigma$ is a global minimum of the Hamiltonian function.
The \emph{stability level} of a configuration $\sigma \not \in \mathcal{X}^s$ is defined as 
\begin{equation}
V_{\sigma}:=\Phi(\sigma,\mathcal{I}_\sigma)-H(\sigma).
\end{equation}
If there are no configurations with energy smaller than $H(\sigma)$, then we set $V_{\sigma}=\infty$. 
We note that the stability level $V_\sigma$ is the minimal energy-cost that, starting from $\sigma$, has to be payed in order to reach states at energy lower than $H(\sigma)$.

To define the set of metastable states, we introduce the \emph{maximal stability level}, 
\begin{equation}\label{Gamma}
    \Gamma:=\max_{\sigma\in \mathcal{X}\setminus \mathcal{X}^s}V_{\sigma}.
\end{equation}

The metastable states are those state that attain the maximal stability level $\Gamma< \infty$, i.e., $\mathcal{X}^m:=\{\sigma \in \mathcal{X}| \, V_{\sigma}=\Gamma \}$.
In \cite{cirillo2013relaxation}, the authors find a relaxation between the maximal stability level of a metastable state \textbf{m} and the communication height between \textbf{m} and a stable state \textbf{s}, i.e.
\begin{align}
    \Gamma=\Phi (\textbf{m}, \textbf{s})-H(\textbf{m}).
\end{align}
For this reason, in the rest of the paper we will use the notation $\Gamma$ to refer indistinctly to the two definitions.

We frame the problem of metastability as the identification of the metastable states 
and the computation of the transition times between them and the stable states. 
To study the transition between $\mathcal{X}^m$ and $\mathcal{X}^s$, 
we define the \emph{first hitting time} of 
$A\subset \mathcal{X}$ starting from $\sigma \in \mathcal{X}$ as
\begin{equation}\label{fht}
    \tau^{\sigma}_A:=\inf\{t>0 \,|\, X_t\in A\}.
\end{equation}
Whenever possible we shall drop the superscript denoting the
starting point $\sigma$ from the notation and we denote by $\mathbb{P}_{\sigma}(\cdot)$ and $\mathbb{E}_{\sigma}[\cdot]$ respectively the probability and the average along the trajectories of the process started at $\sigma$.

\subsection{Main results}\label{subsec:main_results}
In this Section we present some conditions to determine a metastable regime and we state the main results on the estimate of the first hitting time from $\textbf{-1}$ to the unique stable state $\textbf{+1}$.
From now on, we will rely on the following assumptions.

\begin{cond}\label{parameters_conditions}
The parameters of the model satisfy the following properties:
\begin{itemize}
    \item[1.] $n> \max \left \{\lceil r^r \rceil, \left \lceil \frac{r^2J}{h} \right \rceil \right \}$ and $r \geq 3$;
    \item[2.] $h>0$ and $J \geq \max\{1,  \lceil \frac{h}{i(G_n)} \rceil \}$.
\end{itemize}
\end{cond}
We note that the second condition implies that $J>h/r$, since $i(G_n)\leq r$.
Assume these conditions, we are able to prove the next result. 
\begin{lemma}
    The homogeneous state $\textbf{-1}$ is a local minimum but not a global minimum for the Hamiltonian function defined in \eqref{hamiltonian_function}.
\end{lemma}
\begin{proof}
Recalling $\Delta H(\textbf{-1})=0$, we show a configuration $\eta$ such that $\Delta H(\eta)>0$. Let $\eta$ be a configuration with only one plus spin in a site $x\in V_n$. We compute $\Delta H(\eta)$ recalling that $d$ is the diameter of the graph.
\begin{align}\label{lem:energy_minus}
    \Delta H(\eta)& = -2h +2\sum_{i=1}^{d} \mathcal{J}(i) | \partial_e^{(i)} x|
    \geq -2h +2\mathcal{J}(1) r>0,
\end{align}
where the first inequality follows by noting that each site in $G_n$ is at distance one from $r$ other sites. Recalling \eqref{interactions}, we have that $\mathcal{J}(1)=J$ and the positivity follows from Condition \ref{parameters_conditions}-(2). 
Thus, $\textbf{-1}$ is a local minimum but it is not a global minimum indeed by \eqref{energy_plus_minus} we have that $\Delta H(\textbf{+1})<0$. 
\end{proof}

In the following, we provide an estimate for the stability level of the homogeneous state $\textbf{-1}$ and we prove that all the other configurations different from ${\textbf{-1}}$ and ${\textbf{+1}}$ have stability level smaller than the upper bound of $V_{\textbf{-1}}$.
Formally, we define $\Gamma_l \leq \Gamma_u$ as follows
\begin{align}
    &\Gamma_l=J (r/2 -o(r)) n\label{def:gamma_lower} , \\
    &\Gamma_u=J(rf(r,\lambda)+o(r))n \label{def:gamma_upper}.
\end{align}
where
\begin{align}\label{function_f_interaction}
    f(r,\lambda)=
    \begin{cases}
        r \qquad &\text{ if } \mathcal{J}(i)=Jr^{1-i}, \\
        r-2 \qquad &\text{ if } \mathcal{J}(i)=Ji^{-\lambda} \text{ and } \lambda \geq d, \\
        \zeta (\lambda)(r-1)^d \qquad &\text{ if } \mathcal{J}(i)=Ji^{-\lambda} \text{ and } \lambda<d.
    \end{cases}
\end{align}

This implies that the homogeneous state $\textbf{-1}$ is metastable. However, we are not able to determine whether other metastable states exist. It is likely that this depends on the strength of the external magnetic field and the nature of the long-range interaction.

It is observed that the estimate of $V_{\textbf{-1}}$ is similar to the one found in \cite{dommers2017metastability} for the short-range case, when the long-range interaction decays exponentially with distance. This result is intuitive: in such a case, the interaction decays so rapidly that contributions beyond a certain distance become negligible. Effectively, the system perceives only a local neighborhood, and only neighbors at short distances contribute significantly to the energy. 
However, this estimate is less precise than in the short-range case: here the upper bound grows quadratically while the lower bound is only linear, whereas in the short-range case both bounds scale linearly with $r$. Nevertheless, the relevant variable is the graph size $n$ and in terms of $n$ both bounds are linear even in this case.
Also when the long-range interaction is defined by a fractional function with a sufficiently large exponent, it is still possible to obtain a meaningful estimate of the energy barrier. However, the quality of the result deteriorates significantly when the long-range interaction becomes stronger, that is, when the exponent is smaller and the evolution of the system depends substantially on distant sites. In such cases, the local approximation is no longer adequate, and the contribution of long-range interactions becomes dominant.

Formally, assuming Condition \ref{parameters_conditions}, we will prove the following results.
\begin{proposition}[Estimate of $V_{\textbf{-1}}$]\label{condition1_maximal_stability}
The stability level of $\textbf{-1}$ is such that
$\Gamma_l \leq V_{\textbf{-1}} \leq \Gamma_u$.
\end{proposition}
\begin{proposition}[Upper bound of the stability level]\label{condition2_maximal_stability}
For all $\sigma \neq \textbf{+1}$, we have $V_\sigma \leq \Gamma_u$.
\end{proposition}

We postpone the proofs on the previous propositions to Sections \ref{sec:proof1} and \ref{sec:proof2}. Next, we present our main theorem on the metastable time.

\begin{theorem}
Let $G_n$ be a random r-regular graph, then for all $\varepsilon>0$
\begin{align}
    \lim_{\beta \to \infty} \mathbb{P}_{\textbf{-1}} \left (e^{\beta (\Gamma_l-\varepsilon)} < \tau_{\textbf{+1}}<e^{\beta (\Gamma_u +\varepsilon)} \right ) =1.
\end{align}
Moreover, if $\eta$ is a metastable state, then the same holds with $\textbf{-1}$ replaced by $\eta$. 
\end{theorem}
\begin{proof}
By applying \cite[Theorem 4.1]{manzo2004essential} with $\eta_0=\textbf{-1}$ 
we get that
\begin{align}
    \lim_{\beta \to \infty} \mathbb{P}_{\textbf{-1}} \left (e^{\beta(V_{\textbf{-1}}-\varepsilon)} < \tau_{\textbf{+1}}<e^{\beta(V_{\textbf{-1}}+\varepsilon)} \right )=1, 
\end{align}
for any $\epsilon>0$.
Thus, by using Proposition \ref{condition1_maximal_stability}, the result follows by noting that the event $\{e^{\beta(V_{\textbf{-1}}-\varepsilon)}< \tau_{\textbf{+1}}<e^{\beta(V_{\textbf{-1}}+\varepsilon)}\}$
is a subset of the event $\{e^{\beta(\Gamma_l-\epsilon)}< \tau_{\textbf{+1}}<e^{\beta(\Gamma_u+\epsilon)}\}$. 

Let $\eta$ be a metastable state, then we have $V_\eta \geq V_{\textbf{-1}}$ by definition of metastable states. Thus, we conclude by applying Proposition \ref{condition2_maximal_stability} and \cite[Theorem 4.1]{manzo2004essential}.
\end{proof}

\section{Proof of Proposition \ref{condition1_maximal_stability}}\label{sec:proof1}
In this section we find an upper and a lower bound for the stability level of $\textbf{-1}$. In order to estimate it, we consider a set $A \subset V_n$ and a configuration $\sigma$ such that $\sigma_A=\textbf{+1}_A$ and $\sigma_{V_n \setminus A}=\textbf{-1}_{V_n \setminus A}$. We can write $\Delta H(\sigma)$ as 
\begin{align}
    \Delta H(\sigma)&=-2h |A| +2\sum_{i=1}^{d} \mathcal{J}(i) |\partial_e^{(i)} A|, \label{def:H_sigma} 
\end{align}
where $d$ is the diameter of the graph. 

We note that the nucleation is not homogeneous in the space, indeed each site has $r$-nearest neighbors, but the number of vertices at different distance $d$ from a considered site may vary between sites. This implies that flipping a minus to a plus in a site $x$ is advantageous when the site has a small number of minus neighbors at a certain distance $d$ and a greater number of minus neighbors at distances greater than $d$.

\subsection{Lower bound on $V_{\textbf{-1}}$} With the following lemma we will find the lower bound for the stability level of $\textbf{-1}$. 
\begin{lemma}
Let $G_n$ a r-regular graph and assume the parameter Conditions \ref{parameters_conditions}. Let $\Gamma_l$ be defined as in \eqref{def:gamma_lower}. If $0<h<J C \sqrt{r} $ for some positive constant $C< \frac{\sqrt{3}}{2}-\sqrt{\log{2}}$, then w.h.p. $V_{\textbf{-1}} \geq \Gamma_l$.
\end{lemma}
\begin{proof}
    We consider a path $\omega$ from $\textbf{-1}$ to $\textbf{+1}$ and we observe that, recalling Definition \ref{def:manifold}, this path crosses each manifold $\nu_k$ for $k=0,...,n$. In particular $\omega$ crosses a configuration in the manifold $\nu_{\lfloor\frac{n}{2} \rfloor}$. Suppose without loss of generality $n$ even, the other case is similar.
    
    Let $A \subset V_n$ be such that $|A| \leq \frac{n}{2} $. Suppose first that $A$ is connected and let $d \geq 1$ is the diameter of $G_n$. We consider the configuration $\sigma$ such that $\sigma_A=\textbf{+1}_A$, $\sigma_{V_n \setminus A}=\textbf{-1}_{V_n \setminus A}$. 
    Recalling Definition \ref{def:isoperimetric_number}, we note that
\begin{align}\label{estimate_ratio}
    \sum_{i=1}^{d}\frac{|\partial_e^{(i)} A|}{|A|} \mathcal{J}(i) \geq J i(G_n)
\end{align}
for the both types of long-range interactions defined in \eqref{interactions}.
Thus, by using  \eqref{def:H_sigma}, \eqref{estimate_ratio}, and Theorem \ref{thm:lowerbound_isop}, we obtain the following lower bound.
\begin{align}\label{inequality_lower}
        \Delta H(\sigma)&\geq 2|A| \left ( -h+J i(G_n) \right ) \geq n \Big ( -h+J\Big ( \frac{r}{2} -\sqrt{\log{2}} \sqrt{r} \Big ) \Big ).
    \end{align}
    Recalling the assumption on $h$ and $C$, we conclude
    \begin{align}
        \Delta H(\sigma) &\geq n \Big ( - J \Big ( \frac{\sqrt{3}}{2} -\sqrt{\log{2}} \Big )\sqrt{r} +J \Big ( \frac{r}{2} -\sqrt{\log{2}} \sqrt{r}  \Big ) \Big ) = Jn\Big ( \frac{r}{2} - \frac{\sqrt{3}}{2} \sqrt{r}\Big ). \notag
    \end{align}

We note that if $A$ is not connected, then $\partial_e^{(i)} A$ can be empty for some $i \in \{1,...,d\}$. Indeed, in this case we can write $A=\dot{\bigcup}_j A_j$ where $A_j$ is a connected set for each $j$. If each vertex of $A$ is at distance $k$ from another vertex in $A$, then $\partial_e^{(k)} A= \emptyset$. This implies that we can rewrite $\Delta H(\sigma)$ as in \eqref{def:H_sigma} and conclude the proof by arguing as above. 

\end{proof}

\subsection{Upper bound on $V_{\textbf{-1}}$}

We consider $G_n$ a random r-regular graph and let $U \subseteq V_n$. Let $x \in U^c$, we define the set of pairs of vertices composed by $x$ and a vertex in $U$ at distance $i$, i.e.
    \begin{align}\label{def:set_vicini_x_dist_i}
        \partial_e^{(i)} x (U)= \{ (x,y) \, | \, y\in U \text{ such that } d(x,y)=i\}.
    \end{align}

The following result identifies certain structural properties of specific configurations such that, whenever these properties are satisfied, the existence of an alternative configuration with lower energy is ensured. The proof is left to Section \ref{sec:auxiliary_lemma}.
\begin{lemma}\label{lem:upperbound_1}
Let $G_n$ be a r-regular graph and let $d$ be its diameter. Given $A \subseteq V_n$ and $x \in A$, we consider $B=A \setminus \{x\}$. Let $\sigma, \eta$ be two configurations such that $\sigma_{B}=\textbf{+1}_{B}$, $\sigma_{B^c}=\textbf{-1}_{B^c}$, $\eta_A=\textbf{+1}_A$ and $\eta_{A^c}=\textbf{-1}_{A^c}$.
Then, we have $H(\sigma) \leq H(\eta)$ if and only if 
\begin{align}\label{diff_energy_upper bound}
h \geq \sum_{i=1}^{d} \mathcal{J}(i) \left (| \partial_e^{(i)} x (A^c) |-| \partial_e^{(i)} x (A) | \right ).
\end{align}
Moreover, let $z\in A^c$ and we consider $\tilde B=A \cup \{z\}$. Let $\xi$ be a configuration such that $\xi_{\tilde B}=\textbf{+1}_{ \tilde B}$, $\xi_{\tilde B^c}=\textbf{-1}_{\tilde B^c}$. Then $H(\xi) \leq H(\eta)$
if and only if
\begin{align}
h \leq \sum_{i=1}^{d} \mathcal{J}(i) \left (| \partial_e^{(i)} z (A^c) |-| \partial_e^{(i)} z (A) | \right ).
\end{align}
\end{lemma}

Lemma \ref{lem:upperbound_2} (resp. Lemma \ref{lem:upperbound_3}) estimates the energy barrier between a configuration with a number of positive spins less than or equal to (resp. greater than or equal to) half the size of the graph, and another configuration containing a smaller (resp. larger) number of positive spins. We postpone the proofs of these two results in Section \ref{sec:auxiliary_lemma}.

\begin{lemma}\label{lem:upperbound_2}
Let $G_n$ be a connected r-regular graph and assume Condition \ref{parameters_conditions}. Let $A \subset V_n$ with $1 \leq |A| \leq n/2$. Then, for every configuration $\sigma$ such that $\sigma_A=\textbf{+1}_A$ and $\sigma_{V_n \setminus A}=\textbf{-1}_{V_n \setminus A}$, there exists a set $B \subset A$ and a configuration $\eta$ with $\eta_B=\textbf{+1}_B$ and $\eta_{V_n \setminus B}=\textbf{-1}_{V_n \setminus B}$ 
such that $H(\eta)<H(\sigma)$ and 
\begin{align}
\Phi (\sigma, \eta)-H(\sigma) \leq \left (2h +2\sum_{i=1}^{d}\mathcal{J}(i) (r(r-1)^{i-1}-1) \right ) s
\end{align}
    with $s=\left \lceil  |A| \left (1 -\frac{-h+J i(G_n) }{\sum_{i=1}^{d} r(r-1)^{i-1} \mathcal{J}(i)} \right )  \right \rceil$. 
\end{lemma}

\begin{lemma}\label{lem:upperbound_3}
    Let $G_n$ be a connected r-regular graph and assume Condition \ref{parameters_conditions}. Let $A \subseteq V_n$ with $n/2 \leq |A| < n $. Then, for every configuration $\sigma$ such that $\sigma_A=\textbf{+1}_A$ and $\sigma_{V_n \setminus A}=\textbf{-1}_{V_n \setminus A}$, there exists a set $B \supset A$ and a configuration $\eta$ with $\eta_B=\textbf{+1}_B$ and $\eta_{V_n \setminus B}=\textbf{-1}_{V_n \setminus B}$ such that $H(\eta)<H(\sigma)$ and
    \begin{align}
        \Phi (\sigma, \eta)-H(\sigma) \leq \left (2h +2\sum_{i=1}^{d}\mathcal{J}(i) (r(r-1)^{i-1}-1) \right ) s
    \end{align}
    with $s=|A^c| \left \lceil 1 -\frac{h+J i(G_n) }{\sum_{i=2}^{1} r (r-1)^{i-1} \mathcal{J}(i)} \right \rceil$. 
\end{lemma}

Next, we report two technical lemmas, which will serve as tools in proving the main results. The proofs of them are given in Section \ref{sec:auxiliary_lemma}.
\begin{lemma}\label{lem:upper_bound_vicini_x}
Given $d \in \mathbb{N}$, let $c_1,...,c_d \geq 0$ be integer constants such that $\sum_{i=1}^d c_i \geq d$. Assume that if there exists $i \in \{1,...,d\}$ such that $c_i=0$, then $c_{i+1}=0$. Let $f: \mathbb{N} \to \mathbb{R}$ be a strictly decreasing function, thus for all $c=(c_1,...,c_d) \in \mathbb{N}^d$,
\begin{align}
    \sum_{i=1}^d c_i f(i) \geq \sum_{i=1}^d f(i).
\end{align}
\end{lemma}
\begin{lemma}
\label{cost_different_interactions1}
    We have the following upper-bound:
    \begin{align}
        \sum_{i=1}^{d}\mathcal{J}(i) (r(r-1)^{i-1}-1) \leq 
         \begin{cases}
                Jr^2 \qquad & \text{ if } \mathcal{J}(n)=Jr^{1-n},\notag \\
                Jr(r-2) \qquad &\text{ if } \mathcal{J}(n)= J i^{-\lambda} \,\,\,\text{ and } \,\,\,\lambda \geq d, \notag \\
                Jr(r-1)^d \zeta (\lambda)\qquad &\text{ if } \mathcal{J}(n)= J i^{-\lambda} \,\,\,\text{ and } \,\,\,\lambda<d.
            \end{cases}
    \end{align}
\end{lemma}

Based on the previous lemmas, it is now possible to derive an upper bound on the communication height between the homogeneous states $\textbf{-1}$ and $\textbf{+1}$. Formally,
\begin{lemma}\label{lem:upperbound_4}
    Let $G_n$ be a random r-regular graph and assume Condition \ref{parameters_conditions}, then w.h.p. 
    \begin{align}
       \Phi(\textbf{+1}, \textbf{-1})&\leq H(\textbf{-1})+\Gamma_u,
    \end{align}
    where $\Gamma_u$ is defined in \eqref{def:gamma_upper}.
\end{lemma}
\begin{proof}
Let $P \subset V_n$ such that $P \in \mathscr{I}$ where 
    \begin{align}
        \mathscr{I}=\text{arg}\min_{\substack{A \subset V_n \\ |A|=n/2} } \frac{|\partial_e A|}{|A|}.
    \end{align}
Define
\begin{align}
    i'(G_n)=\min_{\substack{A \subset V_n \\ |A|=n/2} } \frac{|\partial_e A|}{|A|}.
\end{align}
We consider the configuration $\sigma \in \mathcal{X}$ such that $\sigma_{P}=+1_{P}$ and $\sigma_{P^c}=-1_{P^c}$. 
From Lemma \ref{lem:upperbound_2}, there exists a set $Q \subset P$, $|Q|<|P|$, and a configuration $\eta$ such that $\eta_{Q}=+1_{Q}$ and $\eta_{Q^c}=-1_{Q^c}$ with $H(\eta)<H(\sigma)$ and 
\begin{align*}
    \Phi(\sigma, \eta)\leq H(\sigma) +\left (2h +2\sum_{i=1}^{d}\mathcal{J}(i) (r(r-1)^{i-1}-1)\right ) 
\left \lceil  \frac{n}{2} \left (1 -\frac{-h+J i(G_n) }{\sum_{i=1}^{d}  r(r-1)^{i-1} \mathcal{J}(i)} \right )  \right \rceil.
\end{align*}

Applying Lemma \ref{lem:upperbound_2} once more, we obtain the existence of a set $R \subset Q$, $|R|<|Q|$, and a configuration $\xi$ such that $\xi_{R}=+1_{R}$ and $\xi_{R^c}=-1_{R^c}$ with $H(\xi)<H(\eta)$ and 
\begin{align*}
    \Phi(\eta, \xi)&\leq H(\eta) +\left (2h +2\sum_{i=1}^{d}\mathcal{J}(i) (r(r-1)^{i-1}-1) \right ) 
\left \lceil  |Q| \left (1 -\frac{-h+J i(G_n) }{\sum_{i=1}^{d}  r(r-1)^{i-1} \mathcal{J}(i)} \right )  \right \rceil \notag \\
    & < H(\sigma) +\left (2h +2\sum_{i=1}^{d}\mathcal{J}(i) (r(r-1)^{i-1}-1) \right ) 
\left \lceil  \frac{n}{2} \left (1 -\frac{-h+J i(G_n) }{\sum_{i=1}^{d} r(r-1)^{i-1} \mathcal{J}(i)} \right )  \right \rceil.
\end{align*}
    We iterate this procedure until we reach $\textbf{-1}$. Since $P \in \mathscr{I}$, by the definition of $\sigma$, we deduce $\Delta H(\sigma) \leq n \left ( -h+J i'(G_n) \right )$ and 
    \begin{align}\label{Phi_minusone}
    \Phi(\sigma, \textbf{-1})\leq &H(\sigma) +\left (2h +2\sum_{i=1}^{d}\mathcal{J}(i) (r(r-1)^{i-1}-1) \right ) 
\left \lceil  \frac{n}{2} \left (1 -\frac{-h+J i(G_n) }{\sum_{i=1}^{d}  r(r-1)^{i-1} \mathcal{J}(i)} \right )  \right \rceil \notag \\
    < &H(\textbf{-1})+ n \left ( -h+J i'(G_n))\right ) \notag \\
    &+\left (2h +2\sum_{i=1}^{d}\mathcal{J}(i) (r(r-1)^{i-1}-1) \right ) 
\left \lceil  \frac{n}{2} \left (1 -\frac{-h+J i(G_n) }{\sum_{i=1}^{d}  r(r-1)^{i-1} \mathcal{J}(i)} \right )  \right \rceil .
    \end{align}
    We analyze the last term.
    \begin{align*}
    \left \lceil  \frac{n}{2} \left (1 -\frac{-h+J i(G_n) }{\sum_{i=1}^{d}  r(r-1)^{i-1} \mathcal{J}(i)} \right )  \right \rceil 
    &\leq \left \{ \frac{n}{2} \left (1 -\frac{-h+J \left (\frac{r}{2} -\sqrt{\log{2}} \sqrt{r} \right ) }{\sum_{i=1}^{d} r(r-1)^{i-1} \mathcal{J}(i)} \right ) +1 \right \} \notag \\
    &\leq \frac{n}{2} \left \{ 1 - \frac{J\left (\frac{r}{2} -\sqrt{\log{2}} \sqrt{r} \right ) }{\sum_{i=1}^{d} r(r-1)^{i-1} \mathcal{J}(i)} \right \},
    \end{align*}
    where the first inequality follows from \eqref{upper_bound_CHeeger}, and the second from the fact that $n$ satisfies Condition \ref{parameters_conditions}-(1).
Thus, we obtain the following upper bound for $\Phi(\textbf{-1},\sigma)$.
\begin{align}
    \Phi(\textbf{-1},\sigma) \leq H(\textbf{-1}) &+ n \left ( -h+J i'(G_n))\right ) \notag \\
    &+n\left (h +\sum_{i=1}^{d}\mathcal{J}(i) (r(r-1)^{i-1}-1) \right ) \left ( 1 - \frac{J\left (\frac{r}{2} -\sqrt{\log{2}} \sqrt{r} \right ) }{\sum_{i=1}^{d} r(r-1)^{i-1} \mathcal{J}(i)} \right ). \notag
\end{align}

Next, we bound $\Phi(\sigma, \textbf{+1})$. Starting from $\sigma$, we define a configuration $\eta'$ by flipping a minus spin in $\sigma$ connected with at least a plus. The number of pluses in $\eta'$ is greater than $n/2$, since $\sigma_P=+1_P$ with $|P|=n/2$, then we can apply Lemma \ref{lem:upperbound_3} and we argue as above in order to obtain the following upper bound.
\begin{align}
   \Phi(\sigma, \textbf{+1}) \leq H(\textbf{-1})&+ n \left ( -h+J i'(G_n))\right ) \notag \\
   &+n\left (-h +\sum_{i=1}^{d}\mathcal{J}(i) (r(r-1)^{i-1}-1)\right ) 
\left \lceil   \left (1 -\frac{h+J i(G_n) }{\sum_{i=1}^{d}  r(r-1)^{i-1} \mathcal{J}(i)} \right )  \right \rceil  \notag \\
\leq \eqref{Phi_minusone}. \notag
\end{align}

Finally, we conclude
        \begin{align}\label{Phi_general}
            \Phi (\textbf{-1},\textbf{+1}) \leq &H(\textbf{-1})+n \left ( -h+J \left (\frac{r}{2} -C \sqrt{r} \right )\right ) \notag \\
            & +n \left (h +\sum_{i=1}^{d}\mathcal{J}(i) (r(r-1)^{i-1}-1) \right ) 
  \left (1 - \frac{J\left (\frac{r}{2} -\sqrt{\log{2}} \sqrt{r} \right ) }{\sum_{i=1}^{d} r(r-1)^{i-1} \mathcal{J}(i)} \right )  \notag \\
  \leq & H(\textbf{-1})+nJ \left (-C \sqrt{r} +\sqrt{\log{2}}\sqrt{r} \right )+n\sum_{i=1}^{d}\mathcal{J}(i) (r(r-1)^{i-1}-1) \notag \\
  &+n \frac{\sum_{i=1}^{d}\mathcal{J}(i)}{\sum_{i=1}^{d}\mathcal{J}(i) r(r-1)^{i-1}} J \left (\frac{r}{2}-\sqrt{\log{2}}\sqrt{r} \right ).
        \end{align}
    
   In the rest of the proof, we distinguish two cases according to the choice of the interaction functions and we compute the final value of $\Phi (\textbf{-1},\textbf{+1})$:
    \begin{itemize}
        \item[(i)] $\mathcal{J}(n)= J r^{1-n}$. By using Lemma \ref{cost_different_interactions1}, we estimate the third term in \eqref{Phi_general} as
        \begin{align}
            n\sum_{i=1}^{d}\mathcal{J}(i) (r(r-1)^{i-1}-1) \leq nJr^2. \notag
        \end{align}
        Next, we compute the last term in \eqref{Phi_general}.
        \begin{align}
            \frac{n\sum_{i=1}^{d}\mathcal{J}(i)}{\sum_{i=1}^{d}\mathcal{J}(i) r(r-1)^{i-1}} J \left (\frac{r}{2}-\sqrt{\log{2}}\sqrt{r} \right ) &=
            \frac{nJ(r^d-1)}{r^{d+1}(r-1) \left ( 1-\left (\frac{r-1}{r} \right )^d\right )} \left (\frac{r}{2}-\sqrt{\log{2}}\sqrt{r} \right ) \notag \\
            &\leq 
            nJ  \left (\frac{1}{2}-\frac{\sqrt{\log{2}}\sqrt{r}}{r} \right ), \notag
        \end{align}
        where in the last inequality we used the fact that $\frac{r^d-1}{r^{d+1}(r-1) \left ( 1-\left (\frac{r-1}{r} \right )^d\right )}$ is a decreasing function of $d \geq 1$.
       Thus, we obtain
        \begin{align}
            \Phi(\textbf{+1}, \textbf{-1})&\leq   H(\textbf{-1})+nJ \left (-C \sqrt{r} +\sqrt{\log{2}}\sqrt{r} \right )+nJr^2
  + nJ \left (\frac{1}{2}-\frac{\sqrt{\log{2}}\sqrt{r}}{r} \right ) \notag \\
&\leq H(\textbf{-1})+n J r^2, \notag
        \end{align}
        where the last inequality holds for some constant $C>\sqrt{\log{2}}-2\sqrt{3}/5$. We conclude
        \begin{align}
            \Phi(\textbf{+1}, \textbf{-1})-H(\textbf{-1})&\leq n J (r^2+o(r)).
        \end{align}
        \item[(ii)] $\mathcal{J}(n)= J i^{-\lambda}$. By using Lemma \ref{cost_different_interactions1},  we estimate the third term in \eqref{Phi_general} as
        \begin{align}
            n\sum_{i=1}^{d}\mathcal{J}(i) (r(r-1)^{i-1}-1) \leq 
            \begin{cases}
                nJr(r-2) \qquad &\text{ if } \lambda \geq d, \notag \\
                nJr(r-1)^d \zeta (\lambda)\qquad &\text{ if } \lambda<d.
            \end{cases}
        \end{align}
        Next, we compute the last term in \eqref{Phi_general}.
        \begin{align}
            \frac{n\sum_{i=1}^{d}\mathcal{J}(i)}{\sum_{i=1}^{d}\mathcal{J}(i) r(r-1)^{i-1}} J \left (\frac{r}{2}-\sqrt{\log{2}}\sqrt{r} \right ) &=
            \frac{nJ\sum_{i=1}^{d}\mathcal{J}(i)}{\sum_{i=1}^{d}\mathcal{J}(i) r(r-1)^{i-1}} \left (\frac{r}{2}-\sqrt{\log{2}}\sqrt{r} \right ) \notag \\
            & \leq nJ \zeta (\lambda) \left (\frac{1}{2}-\frac{\sqrt{\log{2}}\sqrt{r}}{r} \right ). \notag
        \end{align}
        Recalling \eqref{function_f_interaction}, we have
        \begin{align}
            \Phi(\textbf{+1}, \textbf{-1})&\leq   H(\textbf{-1})+ nJ \left (-C \sqrt{r} +\sqrt{\log{2}}\sqrt{r} +rf(r,\lambda)+ \zeta (\lambda) \left (\frac{1}{2}-\frac{\sqrt{\log{2}}\sqrt{r}}{r} \right ) \right ). \notag
        \end{align}
        We conclude by using $C>\sqrt{\log{2}}-2\sqrt{3}/5$,
        \begin{align}
            \Phi(\textbf{+1}, \textbf{-1})&\leq H(\textbf{-1})+nJ(rf(r, \lambda)+o(r)).
        \end{align}
    \end{itemize}
\end{proof}

\section{Proof of Proposition \ref{condition2_maximal_stability}}\label{sec:proof2}
Suppose that $G_n$ is a connected $r$-regular graph and assume conditions \ref{parameters_conditions}. We will prove that, for all $\sigma \neq \textbf{+1}$, $V_{\textbf{+1}} \leq \Gamma_u$. In particular, in Theorem \ref{condition1_maximal_stability}, we proved that $V_{\textbf{-1}} \leq \Gamma_u$. Thus, assume $\sigma \not \in \{\textbf{-1}, \textbf{+1}\}$. By applying Lemmas \ref{lem:upperbound_2}-\ref{lem:upperbound_3}, we obtain
\begin{align}\label{eq_stab_all}
V_\sigma \leq \left (2h +2\sum_{i=1}^{d}\mathcal{J}(i) (r(r-1)^{i-1}-1) \right ) \left \lceil  \frac{n}{2} \left (1 -\frac{-h+J i(G_n) }{\sum_{i=1}^{d} r(r-1)^{i-1} \mathcal{J}(i)} \right )  \right \rceil.
\end{align}
If $G_n$ is a random regular graph with $r \geq 3$, then we use \eqref{upper_bound_CHeeger} and we obtain w.h.p.
\begin{align}
   V_\sigma &\leq \left (h +\sum_{i=1}^{d}\mathcal{J}(i) (r(r-1)^{i-1}-1) \right ) \left \lceil  \left (1 -\frac{-h+J \left ( \frac{r}{2}- \sqrt{\log{2}} \sqrt{r} \right )}{\sum_{i=1}^{d} r(r-1)^{i-1} \mathcal{J}(i)} \right )  \right \rceil n \notag \\
   & \leq \left (2h+\sum_{i=1}^{d}\mathcal{J}(i) (r(r-1)^{i-1}-1) -J\left ( \frac{r}{2}- \sqrt{\log{2}} \sqrt{r} \right ) \right ) n. \notag
\end{align}
We conclude by applying Lemma \ref{cost_different_interactions1} to estimate the second sum.

\section{Proofs of Lemmas to compute the estimate of $V_{\textbf{-1}}$}\label{sec:auxiliary_lemma}

\begin{proof}[Proof of Lemma \ref{lem:upperbound_1}]
    Suppose first that $A \subset V_n$ is a connected set of edges, and let $x \in A$. We consider two configurations $\sigma,\eta $ as in the assumption. Recalling that there is not self-interactions and noting that the two configurations differ for only one spin in $x$ (indeed $\sigma_x=-1$, $\eta_x=+1$ and $\sigma_{V_n \setminus \{x \}}=\eta_{V_n \setminus \{x \}}$). In both cases, $| \partial_e^{(i)} x(A^c) |$ is the number of minuses at distance $i$ from the site $x$, and $| \partial_e^{(i)} x (A) |$ is the number of pluses at distance $i$ from the site $x$.
Recalling \eqref{def:H_sigma}, we obtain
    \begin{align}
        H(\eta)-H(\sigma)&=-2h 
        +2 \sum_{i=1}^{d} \mathcal{J}(i) \left (| \partial_e^{(i)} x (A^c) |-| \partial_e^{(i)} x (A) | \right ) \leq 0, \notag
    \end{align}
    where the last inequality holds for 
    \begin{align}
h \geq \sum_{i=1}^{d} \mathcal{J}(i) \left (| \partial_e^{(i)} x (A^c) |-| \partial_e^{(i)} x (A) | \right ). \notag
    \end{align}
The case where $A \subset V_n$ is not connected is treated analogously. \\
    
 In the following, we prove the second part of the lemma. 
    
  \noindent  Suppose first that $A \subset V_n$ is a connected set of edges, and let $z \in A^c$. We consider a configuration $\xi$ as in the assumption. Then, $| \partial_e^{(i)} z (A^c) |$ is the number of pluses at distance $i$ from the site $z$ in both configurations and 
recalling \eqref{def:H_sigma}, we obtain
    \begin{align}
        H(\eta)-H(\xi)
        &=2h -2 \sum_{i=1}^{d} \mathcal{J} (i) \left (| \partial_e^{(i)} z (A^c) |-| \partial_e^{(i)} z (A) | \right )\leq 0, \notag
    \end{align}
    where the last inequality holds for
    \begin{align}
        h \leq \sum_{i=1}^{d} \mathcal{J} (i) \left (| \partial_e^{(i)} z (A^c) |-| \partial_e^{(i)} z (A) | \right ). \notag
    \end{align}
    The case $A \subset V_n$ not connected is similar.
\end{proof}

\begin{proof}[Proof of Lemma \ref{lem:upperbound_2}]
    Given a set $A \subset V_n$ with $1 \leq |A| \leq n/2$, , let $\sigma$ be the configuration such that $\sigma_A=\textbf{+1}_A$ and $\sigma_{A^c}=\textbf{-1}_{ A^c}$. We construct the set $B$ as follows. Starting from $\sigma$, we flip $s$-times a plus spin to minus choosing at random pluses with at least a minus nearest neighbor. We observe that at least such a plus exists, since the graph is connected and $|A| \geq 1$.
Considering $x_j \in A$ for $j=1,...,s$, at each step the energy can go up at most 
\begin{align*}
    &2h +2  \sum_{i=1}^{d}\mathcal{J}(i) \left ( |\partial_e^i x_j(A)|-|\partial_e^i x_j(A^c)| \right ) \leq 
    2h +2  \sum_{i=1}^{d}\mathcal{J}(i) \left ( r(r-1)^{i-1}-|\partial_e^i x_j(A^c)| \right ), \notag
\end{align*}
where we used the property \eqref{inequality_lower_graph}. 
We note that $|A^c| \geq n/2$, then w.h.p. $\sum_{i=1}^{d} |\partial_e^i x_j(A^c)| \geq n/2>d$, where we used the upper bound \eqref{upperbound_diameter}. Suppose first $A$ connected, we have that if there exists $\tilde i$ such that $\partial_e^{\tilde i} x_j(A^c) =\emptyset$, then $\partial_e^i x_j(A^c) =\emptyset$ for each $i>\tilde i$. Moreover, recalling \eqref{interactions}, we observe that $\mathcal{J}(i)$ is a strictly decreasing function. Thus, the assumptions of Lemma \ref{lem:upper_bound_vicini_x} are satisfied and we obtain
\begin{align}\label{eq_before}
    &2h +2  \sum_{i=1}^{d}\mathcal{J}(i) \left ( r(r-1)^{i-1}-|\partial_e^i x_j(A^c)| \right ) 
    \leq 2h +2  \sum_{i=1}^{d}\mathcal{J}(i) \left ( r(r-1)^{i-1}-1 \right ). 
\end{align}
See Lemma \ref{cost_different_interactions1} for more details according to the choice of the long-range interactions \eqref{interactions}.
We will prove that \eqref{eq_before} holds also when $A$ is not connected. Thus, suppose $A$ not connected, then there exists at least two disjoint subsets $S,T$ such that $A=S \cup T$. For $j=1,...,s$, let $s_{i,j}$ (resp. $t_{i,j}$) be the number of vertices in $\partial_e^{(i)}x_j(A^c)$ with $x_j \in S$ (resp. $x_j \in T$), i.e. 
$s_{i,j}= | \{ \partial_e^{(i)}x_j(A^c) \, | \, x_j \in S \} |$ (resp. $t_{i,j}= | \{ \partial_e^{(i)}x_j(A^c) \, | \, T_j \in S \} |$)
then there exists some $1<a<b\leq d$ such that $s_k,r_k=0$ for each $k=a,a+1,...,b$. Thus
\begin{align}
    \sum_{i=1}^{d}\mathcal{J}(i) |\partial_e^i x_j(A^c)|&=\sum_{i=1}^{d}\mathcal{J}(i) s_{i,j}+\sum_{i=1}^{d}\mathcal{J}(i) t_{i,j} 
     \geq \sum_{i=1}^{a-1}\mathcal{J}(i) s_{i,j}+\sum_{i=1}^{a-1}\mathcal{J}(i) t_{i,j} \geq 2\mathcal{J}(1). \notag
\end{align}
Recalling \eqref{interactions}, we note that
\begin{align}
    \sum_{i=1}^d \mathcal{J}(i)< \sum_{i=1}^\infty \mathcal{J}(i) =
    \begin{cases}
        \frac{r}{r-1} &\qquad \text{ if } \mathcal{J}(i)=r^{1-i},\notag \\
        \zeta(\lambda) &\qquad \text{ if } \mathcal{J}(i)=i^\lambda,
    \end{cases}
\end{align}
where $\zeta(n)$ is the Zeta-function defined as $\sum_{k=1}^\infty \frac{1}{k^n}$. Then 
\begin{align}
   \sum_{i=1}^d \mathcal{J}(i)<2J, \notag
\end{align}
since $r\geq 3$ and $\lambda>2$. Thus, \eqref{eq_before} holds.

After flipping these $s$ spins from $+1$ to $-1$, we examine the remaining plus spins.
If no plus spin remains, we are done, since in this case $\eta=\mathbf{-1}$ and by \eqref{lem:energy_minus} we have $H(\mathbf{-1}) < H(\sigma)$. 
Otherwise, we proceed as follows: we continue flipping plus spins to minus, but at each step we randomly select a site $x$ such that $\sigma(x)=+1$ and 
\begin{align}
    h \leq \sum_{i=1}^{d} \mathcal{J}(i) \left (| \partial_e^{(i)} x (A^c) |-| \partial_e^{(i)} x (A) | \right ). \notag
\end{align}

We repeat this procedure until no such spins remain. By Lemma \ref{lem:upperbound_1}, the energy does not increase at any step. We denote the resulting configuration by $\eta$.
Thus, it follows that 
\begin{align}
    \Phi (\sigma,\eta)-H(\sigma) \leq \left (2h +2\sum_{i=1}^{d}\mathcal{J}(i) (r(r-1)^{i-1}-1) \right ) s. \notag
\end{align} 
See Lemma \ref{cost_different_interactions1} for an explicit value of this estimate according to the different definitions of the long-range interaction $\mathcal{J}(\cdot)$ in \eqref{interactions}.

It only remains to choose $s$ large enough to guarantee $H(\eta) < H(\sigma)$.  By construction,  $|B|\leq |A| -s < |A|$ and
for each site $x$ in $B$ we have
\begin{align}
    h > \sum_{i=1}^{d} \mathcal{J}(i) \left (| \partial_e^{(i)} x (B^c) |-| \partial_e^{(i)} x (B) | \right ). \notag
\end{align}
This implies
\begin{align}
    h > \frac{1}{|B|} \sum_{x\in B} \sum_{i=1}^{d} \mathcal{J}(i) \left (| \partial_e^{(i)} x (B^c) |-| \partial_e^{(i)} x (B) | \right )
    &=\sum_{i=1}^{d}\mathcal{J}(i) \left (\frac{| \partial_e^{(i)} B |}{|B|}- \sum_{x\in B} \frac{| \partial_e^{(i)} x (B) |}{|B|} \right ). \notag
\end{align}
Moreover, recalling \eqref{inequality_lower_graph}, we have
\begin{align}
    \Delta H(\eta) =& 2|B| \Big ( -h+\sum_{i=1}^{d} \frac{|\partial_e^{(i)} B|}{|B|}\mathcal{J}(i) \Big ) 
    < 2|B| \sum_{i=1}^{d} \mathcal{J}(i)  \sum_{x\in B}\frac{| \partial_e^{(i)} x (B) |}{|B|}\notag \\
   < & 2|B|  \sum_{x\in B}\sum_{i=1}^{d} \mathcal{J}(i)\frac{|\partial_e x(B)| (r-1)^{i-1}}{|B|}
   < 2 (|A|-s) \sum_{i=1}^{d} \mathcal{J}(i)  r(r-1)^{i-1}, \notag
\end{align}
since by the property of the graph $|\partial_e x(B)|< r$.

Follows from \eqref{inequality_lower} that we need to choose $s$ such that 
\begin{align}
    2 (|A|-s) \sum_{i=1}^{d} \mathcal{J}(i) r(r-1)^{i-1}  \leq
         2|A| \left ( -h+Ji(G_n) \right ),  \notag
\end{align}
which is equivalent to
\begin{align}\label{estimate_s}
    s &\geq  |A| \left (1 -\frac{-h+J i(G_n) }{\sum_{i=1}^{d}  r(r-1)^{i-1} \mathcal{J}(i)} \right ).
\end{align}
Additional details on \eqref{estimate_s} can be found in Remark \ref{value_s}.
\end{proof}

\begin{remark}\label{value_s}
First of all, we will show that the value of $s$ in \eqref{estimate_s} is strictly greater than zero, i.e.
    \begin{align}\label{eq_s_lower}
        1 -\frac{-h+J i(G_n) }{\sum_{i=1}^{d}  r(r-1)^{i-1} \mathcal{J}(i)} >0.
    \end{align}
Indeed, when we choose $\mathcal{J}(i)=Ji^{-\lambda}$, we have
    \begin{align}
       \sum_{i=1}^{d}  r(r-1)^{i-1} \mathcal{J}(i) >\mathcal{J}(1)r>Ji(G_n)>-h+Ji(G_n), \notag
    \end{align}
    where we used the upper bound of $i(G_n)$ in \eqref{upper_bound_CHeeger}. Otherwise, if $\mathcal{J}(i)=Jr^{1-i}$, we obtain
     \begin{align}
       \sum_{i=1}^{d}  r(r-1)^{i-1} \mathcal{J}(i) &=Jr\sum_{i=1}^{d} \left( \frac{r-1}{r} \right )^{i-1}  =Jr^2 \left (1- \left ( 1-\frac{1}{r} \right )^d \right ) \notag \\
       &\geq Jr^2 \left (1- \left ( 1-\frac{1}{r} \right )^r \right ) \geq J \frac{r^2}{2} >J i(G_n)>-h+Ji(G_n), \notag
    \end{align}
where in the first inequality we used that $d \geq r$, since $G_n$ satisfies Condition \ref{parameters_conditions}-(1). The second inequality follows from
\begin{align}
    1- \left ( 1-\frac{1}{r} \right )^r  \geq \frac{1}{2}, \notag
\end{align}
    since the left term is a decreasing function of $r$ and the limit as $r \to \infty$ is equal to $1-\frac{1}{e}>\frac{1}{2}$.

      Moreover,
    \begin{align}\label{eq_s_upper}
        1 -\frac{-h+J i(G_n) }{\sum_{i=1}^{d}  r(r-1)^{i-1} \mathcal{J}(i)} <1,
    \end{align}
    since the denominator of the fraction is positive and by Condition \ref{parameters_conditions}-(2) also the numerator.
\end{remark}

\begin{proof}[Proof of Lemma \ref{lem:upperbound_3}]
Given a set $A \subset V_n$ with $n/2 \leq |A| <n$, let $\sigma$ be the configuration such that $\sigma_A=\textbf{+1}_A$ and $\sigma_{A^c}=\textbf{-1}_{A^c}$. We construct the set $B$ as follows. 
    Starting from $\sigma$, we flip $s$-times a minus spin to plus choosing at random minuses with at least a plus nearest neighbor. We observe that at least such a minus exists, since the graph is connected and $|A| < n$.
 Considering $x_j \in A^c$ for $j = 1, ..., s$, at each step the
energy can go up at most
\begin{align*}
    -2h +2 \sum_{i=1}^d \mathcal{J}(i) (|\partial_e^{(i)}x_j(A^c)|-|\partial_e^{(i)}x_j(A)|) \leq -2h +2 \sum_{i=1}^d \mathcal{J}(i) (r(r-1)^{i-1}-|\partial_e^{(i)}x_j(A)|),
\end{align*}
where we used the property \eqref{inequality_lower_graph}.

Recalling that $|A| \geq n/2$, then w.h.p. $ \sum_{i=1}^d |\partial_e^{(i)}x_j(A)|\geq n/2 > d$, where we used the upper bound \eqref{upperbound_diameter}. Suppose first $A$ connected, then we observe that if there exists $\tilde i$ such that $\partial_e^{\tilde i} x_j(A) = \emptyset$, then
$\partial_e^{i} x_j(A) = \emptyset$ for each $i > \tilde i$. Moreover, since $\mathcal{J}(i)$ is a strictly decreasing function, the assumptions of Lemma \ref{lem:upper_bound_vicini_x} are satisfied and we obtain
\begin{align}\label{eq_before2}
-2h +2 \sum_{i=1}^d \mathcal{J}(i) (r(r-1)^{i-1}-|\partial_e^{(i)}x_j(A)|) \leq -2h +2 \sum_{i=1}^d \mathcal{J}(i) (r(r-1)^{i-1}-1).
\end{align}
The same argument used in the proof of Lemma \ref{lem:upperbound_2} shows that \eqref{eq_before2} remains valid even if $A$ is not connected.

After changing these $s$ spins from minus to plus, we examine the remaining minuses. 
If there are no minus spins left we conclude, because $\eta=\mathbf{+1}$ and by \eqref{lem:energy_minus} follows that $H(\mathbf{+1}) < H(\sigma)$. 
Otherwise, we proceed as follows.
We keep changing spins from minus to plus, but now at each step we select at random a site $x$ such that $\sigma(x)=-1$ and 
\begin{align}
    h \geq \sum_{i=1}^{d} \mathcal{J}(i) \left (| \partial_e^{(i)} x (A^c) |-| \partial_e^{(i)} x (A) | \right ). \notag
\end{align}
We iterate this procedure until such minus spins do not exist anymore. From Lemma \ref{lem:upperbound_1} it follows that the energy can not go up in any of these steps. We call the remaining configuration $\eta$.
From the above we obtain that 
\begin{align}
    \Phi (\sigma,\eta)-H(\sigma) \leq \left (-2h +2 \sum_{i=1}^d \mathcal{J}(i) (r(r-1)^{i-1}-1)  \right ) s. \notag
\end{align} 
We are left with choosing $s$ large enough to guarantee that $H(\eta) < H(\sigma)$. We note for each site in $B^c$ we have
\begin{align}
    h < \sum_{i=1}^{d} \mathcal{J}(i) \left (| \partial_e^{(i)} x (B^c) |-| \partial_e^{(i)} x (B) | \right ). \notag
\end{align}
This implies
\begin{align}\label{h_minore_Bc}
    h < \frac{1}{|B^c|} \sum_{x\in B^c} \left (  \sum_{i=1}^{d} \mathcal{J}(i) \left (| \partial_e^{(i)} x (B^c) |-| \partial_e^{(i)} x (B) | \right ) \right )
    & =\sum_{i=1}^{d} \mathcal{J}(i) \left (  \sum_{x\in B^c} \frac{| \partial_e^{(i)} x (B^c) |}{|B^c|}-\frac{| \partial_e^{(i)} B^c |}{|B^c|} \right ).
\end{align}
Moreover, by construction, we have $|B|\geq |A| +s$, i.e. $|B^c|\leq |A^c|-s$, and 
\begin{align}
 \Delta H(\eta) &= -2h|B|+2|B^c| \sum_{i=1}^{d} \frac{|\partial_e^{(i)} B|}{|B^c|}\mathcal{J}(i) \notag \\
 &= 2|B^c| \left (h+\sum_{i=1}^{d} \frac{|\partial_e^{(i)} B|}{|B^c|}\mathcal{J}(i) \right )-2hn \notag \\
 & \leq 2(|A^c|-s) \left (\sum_{i=1}^{d} \mathcal{J}(i) \left (  \sum_{x\in B^c} \frac{| \partial_e^{(i)} x (B^c) |}{|B^c|}-\frac{| \partial_e^{(i)} B^c |}{|B^c|} \right )+\sum_{i=1}^{d} \frac{|\partial_e^{(i)} B|}{|B^c|}\mathcal{J}(i) \right )-2hn, \notag
\end{align}
where we used \eqref{h_minore_Bc}. By Definition \ref{def:external_boundary}, follows that $\sum_{i=1}^{d}|\partial_e^{(i)} B|=\sum_{i=1}^{d}|\partial_e^{(i)} B^c|$. Thus, we obtain
\begin{align}
     \Delta H(\eta) &\leq 2(|A^c|-s) \left (\sum_{i=1}^{d} \mathcal{J}(i)\sum_{x\in B^c} \frac{| \partial_e^{(i)} x (B^c) |}{|B^c|} \right )-2hn \notag \\
   & \leq -2hn+ 2(|A^c|-s) \sum_{i=1}^{d} \mathcal{J}(i)r (r-1)^{i-1}. \notag
\end{align}
Moreover, recalling $|A| \geq n/2$ and \eqref{def:isoperimetric_number}, we note that 
\begin{align}
    \Delta H(\sigma)&=-2h|A| +2 \sum_{i=1}^d |\partial^{(i)}_e A| \mathcal{J}(i)=2|A^c| \left ( -h \frac{|A|}{|A^c|} + \sum_{i=1}^d \frac{|\partial^{(i)}_e A|}{|A^c|} \mathcal{J}(i) \right ) \notag \\
    &\geq 2|A^c| \left ( -h \frac{|A|}{|A^c|} + Ji(G_n) \right ). \notag
\end{align}
Hence, it follows from \eqref{inequality_lower} that we need to choose $s$ such that 
\begin{align}
    -2hn+ 2(|A^c|-s) \sum_{i=1}^{d} \mathcal{J}(i)r (r-1)^{i-1}\leq
         2|A^c| \left ( -h \frac{|A|}{|A^c|} + Ji(G_n) \right ), \notag
\end{align}
which is equivalent to
\begin{align}
    s \geq  |A^c| \left (1 -\frac{h+J i(G_n)}{\sum_{i=1}^{d} r (r-1)^{i-1} \mathcal{J}(i)}\right ).
\end{align}
\end{proof}

\begin{proof}[Proof of Lemma \ref{lem:upper_bound_vicini_x}]
First, note that if $c_i=1$ for every $i \in \{1,...,d\}$, the result is immediate. Hence, let us consider $(k_1,...,k_d) \neq (1,...,1)$ and prove that
    \begin{align}
        \sum_{i=1}^d k_i f(i)> \sum_{i=1}^d f(i).
    \end{align}
    Since $(k_1,...,k_d) \neq (1,...,1)$, there exists a nonempty subset of indices $\{j_1,...,j_m\} \subseteq \{1,...,d\}$ such that $k_{j_i}>1$ with $n\geq 1$. 
    This subset cannot be empty, otherwise $(k_1,...,k_d)$ would be either $ (1,...,1)$, or $(1,1,1,...,0,0,0)$, or $(0,...,0)$ and these are contradictions since $\sum_{i=1}^d k_i \geq d$ for assumption.
    
    Moreover, if $k_i \geq 1$ for every $i \in \{1,...,d\}$, we conclude. Otherwise, there exists another subset $\{z_1,...,z_{\tilde m}\} \subseteq \{1,...,d\}$ such that $k_{z_i}=0$. In this case,
    \begin{align}
        \sum_{i=1}^d k_i f(i) &= \sum_{i \in \{1,...,d\} \setminus\{z_1,...,z_{\tilde m}\}} k_i f(i)> \sum_{i \in \{1,...,d\} \setminus\{z_1,...,z_{\tilde m},j_1,...,j_m\}} f(i) + \sum_{i \in \{j_1,...,j_m\}} 2f(i) \notag \\
        &> \sum_{i \in \{1,...,d\} \setminus\{z_1,...,z_{\tilde m},j_1,...,j_m\}} f(i) + \sum_{i \in \{j_1,...,j_m\}} f(i)+\sum_{i \in \{ z_1,...,z_{\tilde m}\}}  f(i),
        =\sum_{i=1}^d f(i) \notag
    \end{align}
    where the last inequality follows from
    \begin{align}
        \sum_{i \in\{j_1,...,j_m\}} 2f(i) \geq \sum_{i \in \{j_1,...,j_m\}} f(i)+\sum_{i \in \{z_1,...,z_{\tilde m}\}} f(i). \notag
    \end{align}
    Indeed, by assumption, $k_{z_i}=0$ implies $k_{z_i+t}=0$ for each $t>0$, then $j_i<k_{\tilde i}$ for every $i \in \{1,...,m\}$, $\tilde i \{1,..., \tilde m\}$. 
    Thus, we can conclude since the function $f(i)$ is strictly decreasing.
\end{proof}

\begin{proof}[Proof of Lemma \ref{cost_different_interactions1}]
Based on the definition of the long-range interactions in \eqref{interactions}, we distinguish two cases:
\begin{itemize}
    \item[(i)] $\mathcal{J}(n)=Jr^{1-n}$. In this case, we obtain
        \begin{align}
    \sum_{i=1}^{d}\mathcal{J}(i) (r(r-1)^{i-1}-1)&=J\sum_{i=1}^{d}r^{1-i} \left (r (r-1)^{i-1}-1 \right ) 
     \notag \\ &\leq J \sum_{i=1}^{\infty}r^{1-i} \left ( r(r-1)^{i-1}-1 \right ) \notag \\
     &=Jr\frac{r^2-r-1}{r-1} \leq Jr^2.
    \end{align}
    \item[(ii)] $\mathcal{J}(n)=Jn^{-\lambda}$. In this case, we have
    \begin{align}
    \sum_{i=1}^{d}\mathcal{J}(i) (r(r-1)^{i-1}-1)=J\sum_{i=1}^{d}i^{-\lambda} \left ( r(r-1)^{i-1}-1 \right ). 
    \end{align}
    Here we further distinguish between the cases: $\lambda <d$ and $\lambda \geq d$.
\vspace{0.3cm}
    
    \paragraph{\textbf{CASE $\lambda < d$.}}
        \begin{align}
            J \sum_{i=1}^{d}i^{-\lambda} \left ( r(r-1)^{i-1}-1 \right ) & \leq J (r(r-1)^d-1) \sum_{i=1}^{d}i^{-\lambda} \leq Jr(r-1)^d \sum_{i=1}^{\infty}i^{-\lambda} \notag \\
            &=Jr(r-1)^d \zeta (\lambda).
        \end{align}
    \paragraph{\textbf{CASE $\lambda \geq d$.}} 
        \begin{align}\label{eq:estimate_second_interaction1}
    J \sum_{i=1}^{d}i^{-\lambda} \left ( r(r-1)^{i-1}-1 \right ) &< \frac{Jr}{r-1}  \sum_{i=1}^{d}i^{-\lambda} (r-1)^{i} \notag \\
    &=\frac{Jr}{r-1} \left ( \sum_{i=1}^{r-2}i^{-\lambda} (r-1)^{i} + \sum_{i=r-1}^{d}i^{-\lambda} (r-1)^{i} \right ) \notag \\
    &\leq\frac{Jr}{r-1} \left ( (r-1)(r-3) + \sum_{i=r-1}^{d}i^{-d} (r-1)^{i} \right ),
    \end{align}
    where we used the assumption of $\lambda$. Moreover, we note that each term $i^{-d}$ of the sum is smaller than $(r-1)^{-d}$, thus we obtain
    \begin{align}
    \eqref{eq:estimate_second_interaction1}&<\frac{Jr}{r-1} \left ( (r-1)(r-3) + \sum_{i=r-1}^{d}(r-1)^{i-d} \right ) \notag \\
    &=\frac{Jr}{r-1} \left ( (r-1)(r-3) + \frac{r-1-(r-1)^{r-1-d}}{r-2} \right ) \notag \\
    &<\frac{Jr}{r-1} \left ( (r-1)(r-3) + 1 \right ) <Jr(r-2).
    \end{align}
    \end{itemize}
\end{proof}

\bibliographystyle{abbrv}
\bibliography{references}

\end{document}